\documentclass{amsart}

\usepackage{amsmath, amsthm}%
\usepackage{amsfonts}%
\usepackage{amssymb}%
\usepackage{mathtools}
\usepackage{graphicx}
\usepackage{amscd, amstext, hyperref}
\usepackage{pictexwd}
\usepackage{dcpic}
\usepackage[mathscr]{eucal}
\usepackage{multicol}
\usepackage{multirow}
\usepackage{color}
\usepackage{bm}
\usepackage{makecell} 
\usepackage[section]{placeins} 

\usepackage{genyoungtabtikz} 

\newtheorem{theorem}{Theorem}[section]

\newtheorem{proposition}[theorem]{Proposition}

\newtheorem{conjecture}[theorem]{Conjecture}
\newtheorem{hypothesis}[theorem]{Hypothesis}

\newtheorem*{maintheorem*}{Main Theorem}
\newtheorem*{corollary*}{Corollary}
\newtheorem*{lemma*}{Lemma}
\newtheorem*{keylemma*}{Key Lemma}

\theoremstyle{definition}
\newtheorem{definition}[theorem]{Definition}
\newtheorem{example}[theorem]{Example}
\newtheorem{counterexample}[theorem]{Counterexample}

\theoremstyle{remark}
\newtheorem{remark}[theorem]{Remark}

\numberwithin{equation}{section}

\newcommand{\etalchar}[1]{$^{#1}$}

\newcommand{\Z}{\mathbb{Z}}

\newcommand{\C}{\mathbb{C}}

\newcommand{\GL}{\mathrm{GL}}

\begin{document}
	
	\title{Deformations of highly symmetric Calabi-Yau Grassmannian hypersurfaces}
	\author{Adriana Salerno}
	\thanks{For the first author, this material is based
		upon work supported by and while serving at the National Science Foundation.
		Any opinion, findings, and conclusions or recommendations expressed
		in this material are those of the authors and do not necessarily reflect the
		views of the National Science Foundation.}
	\author{Ursula Whitcher}
	\thanks{The first and second authors thank the Isaac Newton Institute
		for Mathematical Sciences, Cambridge, for support and hospitality during
		the K-theory, algebraic cycles and motivic homotopy theory programme, where work on this paper was undertaken. This work was supported by EPSRC
		grant no EP/R014604/1.}
	\author{Chenglong Yu}
	\thanks{All authors thank Enrico Fatighenti for enlightening mathematical and computational discussions.}
	
	\begin{abstract}
	We use arithmetic and Hodge-theoretic techniques to study pencils of Calabi-Yau varieties realized as highly symmetric hypersurfaces in Grassmannians and their quotients, demonstrating that their geometric properties are distinct from the classical mirrors of Calabi-Yau Grassmannian hypersurfaces.
	\end{abstract}
	
	\maketitle
	
	\section{Introduction}
	
	In \cite{CDK}, Coates, Doran, and Kalashnikov described a candidate construction for the mirrors of Calabi-Yau hypersurfaces in Grassmannians $G(r,n)$. The construction is analogous to the Greene--Plesser mirror for Calabi-Yau hypersurfaces in projective space: it involves the quotient of a pencil of hypersurfaces by a discrete group $H$, together with a resolution of singularities.
	
	We study the topology and arithmetic of the Coates--Doran--Kalashnikov pencil, which we term the \emph{arrow pencil}, in order to compare the \cite{CDK} candidate mirror with other Grassmannian hypersurface mirror constructions. We employ a range of computational techniques to identify properties of the holomorphic periods and their deformations, using both arithmetic and Hodge-theoretic methods. We begin with experiments that offer an unusual application of arithmetic data to gather evidence concerning a geometric mirror symmetry proposal. We then develop effective techniques for studying the middle cohomology of Calabi-Yau Grassmannian hypersurfaces based on work of \cite{FM}. Our results show that the periods of the arrow pencil exhibit more complicated behavior than classical Grassmannian hypersurface mirrors, even for very small $r$ and $n$. The discrepancy arises because the arrow pencils can be extended to multiparameter families of Calabi-Yau hypersurfaces in a natural way. We further demonstrate that other highly symmetric pencils with the same abelian symmetry group admit similar extensions to multiparameter families.
	
	Our main theorem characterizes a subspace of the middle cohomology of certain highly symmetric Calabi-Yau threefold Grassmannian pencils fixed by the Coates--Doran--Kalashnikov group action:
	
\begin{maintheorem*}\label{T:main}
	Let $Z^{2,4}_t$ be any of the following $H_{4,2}$-invariant pencils of Calabi-Yau threefolds in the Grassmannian $G(2,4)$: 
	
	\begin{align*}
		t \left( p_{12}^4 + \cdots + p_{34}^4  \right) + p_{12} p_{23} p_{34} p_{14} & = 0 \\
		t \left( p_{12}^{4}  + \cdots + p_{34}^4  +  p_{14}^{2} p_{23}^{2} +   p_{13}^{2} p_{24}^{2} +  p_{12}^{2} p_{34}^{2} \right) + p_{12} p_{14} p_{23} p_{34} &= 0\\
		t \left( p_{12}^{4} + \cdots + p_{34}^4 +  p_{13} p_{14} p_{23} p_{24} + p_{12} p_{13} p_{24} p_{34}\right) + p_{12} p_{14} p_{23} p_{34} & = 0 \\
		t \left(p_{12}^{4} + \cdots + p_{34}^4 + p_{14}^{2} p_{23}^{2} +  p_{13} p_{14} p_{23} p_{24} + p_{13}^{2} p_{24}^{2} + p_{12} p_{13} p_{24} p_{34} +  p_{12}^{2} p_{34}^{2} \right) & \\ 
		+ p_{12} p_{14} p_{23} p_{34} & = 0.
	\end{align*}
	
	Then $(H^{2,1}(Z^{2,4}_t))^{H_{4,2}}$ is a 5-dimensional vector space.
\end{maintheorem*}

	This theorem yields a counterexample to the Coates--Doran--Kalashnikov mirror proposal for Calabi-Yau hypersurfaces in $G(2,4)$, and shows that certain natural modifications of the proposal will not suffice to yield a mirror. We use similar techniques to establish a counterexample for Calabi-Yau hypersurfaces in $G(2,5)$. These are the lowest dimensions where the proposal might apply.
	
\begin{counterexample}
	The pencil of Calabi-Yau varieties $Y^{r,n}_t$ described in Conjecture~\ref{Con:CDKmirror} is not mirror to Calabi-Yau hypersurfaces in $G(r,n)$ for $r=2$ and $n=4$ or $n=5$.
\end{counterexample}
	
	The paper is organized as follows. In Section~\ref{S:mirrorPencils}, we review classical constructions for mirrors of Calabi-Yau hypersurfaces in projective space and Grassmannians, and describe the arrow pencil construction. In Section~\ref{S:G24}, we compare the arithmetic of the arrow pencil in $G(2,4)$ to classical predictions, using a combination of experimental techniques and work of \cite{HLYY} that relates periods and point counts. We explain the discrepancy in behavior via complete intersection techniques. We use techniques of \cite{FM} in Section~\ref{S:griffithsRing} to extend the observations we made on the arrow pencil in $G(2,4)$ to methods applicable in higher dimension. We apply these techniques to study other $H$-invariant pencils of Calabi-Yau threefolds.
	
	Code associated with this paper is available in the Zenodo repository \cite{code}.
	
	\section{Mirror constructions and Calabi-Yau pencils}\label{S:mirrorPencils}
	
	\subsection{The Greene--Plesser mirror quintic}
	
	We begin by recalling the properties of the Greene--Plesser (\cite{GP}) mirror construction for Calabi-Yau varieties realized as quintic hypersurfaces in $\mathbb{P}^4$. Smooth quintics in $\mathbb{P}^4$ have $h^{2,1} = 101$ and $h^{1,1} = 1$; thus, the mirror quintics should have $h^{2,1} = 1$ and $h^{1,1} = 101$. The mirror construction begins with the Fermat pencil $X_\psi$ given by $x_0^5 + \cdots +x_4^5 - 5 \psi x_0\cdots x_4$. Consider the group $\widetilde{H}$ defined as follows.
	
	\begin{equation}
		\widetilde{H} = \langle (\zeta_0, \dots, \zeta_4) \mid \zeta_i^5 = 1, \prod \zeta_i  = 1 \rangle.
	\end{equation}
	
	The group $\widetilde{H} \cong (\mathbb{Z}/5\mathbb{Z})^4$ acts on each member of the Fermat pencil by a diagonal action. Multiplying every coordinate by the same value has trivial action in projective space. Taking the quotient of $\widetilde{H}$ by elements of the form $(\zeta, \dots, \zeta)$, we obtain a faithful action by a group $H \cong (\mathbb{Z}/5\mathbb{Z})^3$. The mirror family of quintics $X^\circ_\psi$ is obtained by taking the quotient $X_\psi/H$ and resolving singularities. In particular, by studying the classes introduced by resolution of singularities, one can show that $X^\circ_\psi$ has $h^{1,1} = 101$.
	
	\subsection{Classical mirrors of Calabi-Yau hypersurfaces in Grassmannians}
	
	A smooth degree $n$ hypersurface $X$ in the Grassmannian variety $G(r,n)$ is a Calabi-Yau variety of dimension $e \coloneqq n(n-r)-1$. The Grassmannian has a hyperplane class, determined by its Pl\"ucker embedding, also denoted as the Schubert class $\sigma_1$, and by intersecting with the hyperplane class we see that the Calabi-Yau hypersurface $X$ has $h^{1,1}(X) = 1$. The classical mirror $X^\circ$ of $X$ must therefore satisfy $h^{e-1,1}(X^\circ)=1$.
	
	The smallest example that does not reduce to a hypersurface in projective space is the case of degree~4 hypersurfaces in $G(2,4)$; in this case, the dimension of the Grassmannian $G(2,4)$ is $2(4-2) = 4$, so these hypersurfaces are Calabi-Yau threefolds. Of course, their Hodge numbers satisfy $h^{1,1}(X) = 1$; in this case, $h^{2,1}(X) = 89$. The Grassmannian $G(2,4)$ is itself embedded in the projective space $\mathbb{P}^5$ by a single Pl\"{u}cker relation, so we may also view the Calabi-Yau threefolds as complete intersections in $\mathbb{P}^5$ determined by one equation of degree~4 and one equation of degree~2. Using this description, Libgober and Teitelbaum showed in \cite{LT} that the mirror should be a one-parameter family with holomorphic periods satisfying a hypergeometric equation of the form ${}_4F_3\left( \frac{1}{4}, \frac{1}{2}, \frac{3}{4}, \frac{1}{2}; 1, 1, 1\right)$.
	
	The authors of \cite{conifold} used a series of geometric transitions to relate Calabi-Yau complete intersections in Grassmannians to complete intersections in toric varieties. From here, one may construct a mirror using standard toric techniques. In general, one would expect the periods of Calabi-Yau complete intersections in toric varieties to satisfy a GKZ or $A$-hypergeometric system of differential equations, but in several examples the authors of \cite{conifold} are able to extract classical hypergeometric equations by strategic specialization of parameters. Though their work is focused on threefolds, the same techniques may be used to study higher-dimensional Calabi-Yau varieties.
	
	\subsection{Arrow pencils and a candidate mirror}
	
	We describe the pencils we will study, following \cite[\S 3]{CDK}, and then recall their mirror proposal.
	
	Set $k = n-r$. We may specify a Pl\"ucker coordinate on $G(r,n)$ in one of two ways. First, we may give an $r$-tuple of indices in $\{1,\dots,n\}$. Alternatively, we may choose a partition whose Young diagram fits inside an $r \times k$ grid; we take the convention that the Young diagram fits into the upper left-hand corner. We convert from a partition to an $r$-tuple of indices by starting at the lower left-hand corner of our grid and traveling to the upper right-hand corner along the edge of the Young diagram, always moving either right or up. Such a path must have $n$ segments; we record the indices of the segments that are vertical.
	
	\begin{example}\label{E:excesspartition}
		If $r=2$ and $n=5$, the partition $2+1$ corresponds to the 2-tuple $(2,4)$ and the Pl\"ucker coordinate $p_{2,4}$.
		
		$$
		
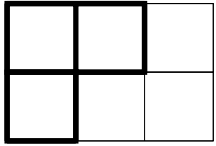
\begin{tikzpicture}[scale=2]
			\Yfillopacity{.5}
			\tyng(0cm,0cm,3,3)
			\Yfillopacity{0}
			\Ylinethick{2pt}
			\tyng(0cm,0cm,2,1)
		\end{tikzpicture}
		$$
	\end{example}
	
	The \emph{frozen variables} are the Pl\"ucker coordinates corresponding to one of the $n$ cyclic $r$-tuples of indices $(1,\dots,r)$, $(2,\dots,r+1)$, \dots, $(n,\dots,r-1)$. In terms of partitions, $(1,\dots,r)$ corresponds to the empty partition, the rectangles of height $r$ and width at most $k-1$ correspond to $(2,\dots,r+1)$, \dots, $(n-r,\dots, n-1)$, and the rectangles of width $k$ correspond to $(n-r+1,\dots,n)$, \dots, $(n,\dots,r-1)$. In other words, we may view the frozen variables as corresponding to partitions of either full width or full height, together with the empty partition.
	
	The arrow partitions $\mathcal{A}(r,k)$ generalize the rectangular partitions used to make frozen variables by allowing either a partial row or a partial column, with the following restrictions on overall partition size.
	
	\begin{definition}
		We say a partition whose Young diagram fits inside an $r \times k$ grid is an \emph{arrow partition} if it has one of the following forms:
		\begin{enumerate}
			\item The empty partition.
			\item $r$ rows of length $k$.
			\item $0$ to $r-2$ rows of length $k$ together with one row of length $a$, $1 \leq a \leq k$.
			\item $1$ to $k-1$ columns of height $r$ together with one column of height $b$, $0 \leq b \leq r-1$.
		\end{enumerate}
		We refer to the corresponding Pl\"ucker coordinates as \emph{arrow variables}. Partitions that are not arrow partitions are called \emph{excess partitions}.
	\end{definition}
	
	There are $2 + (r - 1) (n - r) + (n - r - 1) r = 2(r - 1)(n - r - 1) + n$ arrow partitions.
	
	\begin{definition}
		The \emph{arrow pencil} of Calabi-Yau hypersurfaces $X^{r,n}_t$ in the Grassmannian $G(r,n)$ is given by the equation
		\begin{equation*}
			t \left(\sum_{\lambda \in \mathcal{A}(r,k)} p_\lambda^n\right) + \prod_{p_\mu \text{ frozen}} p_\mu = 0.
		\end{equation*}
	\end{definition}
	
	\begin{remark}
		The authors of \cite{CDK} use the parameter $\psi = 1/t$.
	\end{remark}
	
	\begin{example}
		For $r=2, n=4$, there are $4\choose 2$ Pl\"ucker coordinates and $2(2 - 1)(4 - 2 - 1) + 4 =6$ arrow partitions, so every Pl\"ucker coordinate is an arrow variable. The arrow pencil is given by:
		\[ t \left( p_{12}^4 + p_{13}^4 + p_{14}^4 + p_{23}^4 + p_{24}^4 + p_{34}^4  \right) + p_{12} p_{23} p_{34} p_{14} = 0.\]
		
		For $r=2, n=5$, there are ${5\choose 2}=10$ Pl\"ucker coordinates and $2(2 - 1)(5 - 2 - 1) + 5 =9$ arrow partitions. The only excess partition is $3+1$, which corresponds to the Pl\"ucker coordinate $p_{2,5}$. The arrow pencil for $G(2,5)$ is thus given by:
		\[ t \left( p_{12}^5 + p_{13}^5 + p_{14}^5 + p_{15}^5 + p_{23}^5 + p_{2,4}^5 + p_{34}^5 + p_{35}^5 + p_{45}^5 \right) + p_{12} p_{23} p_{34} p_{45} p_{15} = 0.\]
	\end{example}

	Following \cite{CDK}, we now define an abelian group $H_{n,r}$ that acts on each member $X^{r,n}_t$ of the arrow pencil. The construction generalizes the diagonal group action on the Fermat pencil. We begin by observing that $\GL(\C, n)$ acts on the $r \times n$ matrices by right multiplication. Define a subgroup of the diagonal matrices in $\GL(\C, n)$ by
	
	\begin{equation}
		\widetilde{H}_{n,r} = \langle (\zeta_1, \dots, \zeta_n) \mid \zeta_i^n = 1, \left(\prod \zeta_i \right)^r = 1 \rangle.
	\end{equation}
	
	The action of $\widetilde{H}_{n,r}$ sends a Pl\"ucker coordinate $p_{i_1\dots i_r}$ to $\zeta_{i_1} \dots \zeta_{i_r} p_{i_1\dots i_r}$, so it fixes each of the monomials in the arrow pencil. 
	
	\begin{definition}
		The finite group $H_{n,r}$ is the subgroup of $\widetilde{H}_{n,r}$ that acts effectively on the Grassmannian $G(r,n)$.
	\end{definition}
	
	\begin{example}\label{ex:Hgroups}
		The group $\widetilde{H}_{4,2}$ is a subgroup of the group generated by four primitive fourth roots of unity that is isomorphic to $(\Z/(4))^3 \times \Z/(2)$. Multiplying every Pl\"ucker coordinate by the same factor fixes the Grassmannian $G(2,4)$, so $H_{4,2}$ is the group of order 32 isomorphic to $(\Z/(4))^2 \times \Z/(2)$.
		
		The group $\widetilde{H}_{5,2}$ is a subgroup of the group generated by five primitive fifth roots of unity that is isomorphic to $(\Z/(5))^4$. The group $H_{5,2}$ is isomorphic to $(\Z/(5))^3$.
	\end{example}
	
	We may now give a precise description of the \cite{CDK} candidate mirror.
	
	\begin{conjecture}[\cite{CDK}]\label{Con:CDKmirror}
		The arrow pencil $X^{r,n}_t$ defines a pencil of Calabi-Yau hypersurfaces $Y^{r,n}_t$ in a resolution of singularities of $G(r,n)/H_{n,r}$ that is mirror to Calabi-Yau hypersurfaces in $G(r,n)$.
	\end{conjecture}
	
	\section{The arrow pencil in $G(2,4)$}\label{S:G24}
	
	In this section, we make a detailed study of the arrow pencil in the case of Calabi-Yau threefold hypersurfaces in $G(2,4)$ and show that its properties differ from the classical mirror, yielding a counterexample to Conjecture~\ref{Con:CDKmirror}. We begin by gathering arithmetic evidence, then explain the discrepancy between the arrow pencil and the classical mirror using Hodge-theoretic techniques.
	
	\subsection{Finite field data and the Hasse-Witt invariant}

	The Picard-Fuchs equation of a pencil of Calabi-Yau varieties is preserved under quotient by a finite group preserving the holomorphic form followed by resolution of singularities. Thus, if Conjecture~\ref{Con:CDKmirror} holds, we should expect that the Picard-Fuchs equation of the arrow pencil $X^{2,4}_t$ is consistent with the mirror Picard-Fuchs equation ${}_4F_3\left( \frac{1}{4}, \frac{1}{2}, \frac{3}{4}, \frac{1}{2}; 1, 1, 1\right)$ computed in \cite{LT}, up to a change of variables. We can convert this expectation to an arithmetic hypothesis using the \emph{Hasse-Witt invariant}. Let us recall the properties of this invariant.
	
	Suppose that $X$ is a smooth variety over a finite field $\mathbb{F}_q$, where $q=p^a$ for some prime $p$. The Frobenius operator $\mathcal{F}$ induces a $p$-linear map from $H^n(X, \mathcal{O}_X)$ to itself. The \emph{Hasse-Witt matrix} is a matrix for this map. In the case that $X$ is a Calabi-Yau variety, $H^n(X, \mathcal{O}_X)$ is one-dimensional, so the Hasse-Witt matrix is a $1 \times 1$ matrix. We thus refer to the entry $\mathrm{HW}_q(X)$ in this matrix, an element of $\mathbb{F}_q$, as the Hasse-Witt invariant of $X$. It follows from work of Katz on the zeta function, \cite{katzCongruence}, that for a Calabi-Yau variety over a finite field, its point count modulo $p$, $\# X(\mathbb{F}_q) \pmod{p}$, is entirely controlled by its Hasse-Witt invariant.
	
	Generalizing from observations of Katz and Dwork (cf. \cite{padic}), one might also expect to find a close relationship between the Hasse-Witt invariant and truncations of series solutions of the Picard-Fuchs equation for families of Calabi-Yau varieties. Such a relationship has been established in many cases. For the Fermat quintic pencil in $\mathbb{P}^n$ and the mirror quintic pencil, see \cite{CORV, CORV2}. Adolphson and Sperber studied the Hasse-Witt matrices of Calabi-Yau hypersurfaces in projective space in greater generality in \cite{ASgkz}. For results on Calabi-Yau hypersurfaces in toric varieties, see \cite{vlasenko, BVi, BVii, BViii}, 
	the work of the third author with collaborators in \cite{HLYY}, and work of the first and second authors in \cite{SW}. 
	
	In \cite[Theorem 6.2]{HLYY}, the third author and collaborators further established a formula for the Hasse-Witt invariant in terms of a truncation of the series solution of the Picard-Fuchs equation in the case of Calabi-Yau hypersurfaces in flag varieties. We expect that the Picard-Fuchs equation of the arrow pencil $X_t^{2,4}$ will match the Picard-Fuchs equation of the classical mirror to Calabi-Yau hypersurfaces in the Grassmannian $G(2,4)$, up to a change of variables that preserves the point of maximally unipotent monodromy. Assuming Conjecture~\ref{Con:CDKmirror} holds, we form the following hypothesis:
	
	\begin{hypothesis}\label{hyp:hypergeometricTruncation}
	The number of points on a smooth member of the arrow pencil $X_t^{2,4}$ over a finite field $\mathbb{F}_p$ will satisfy a truncation relationship of the form
		\[\# X_t^{2,4}(\mathbb{F}_p) \equiv 1- {}_4F_3\left( \frac{1}{4}, \frac{1}{2}, \frac{3}{4}, \frac{1}{2}; 1, 1, 1 \mid a t^{b}\right)\pmod{p}. \]
		Here we take $t$ to be a rational number whose denominator is not a multiple of $p$ and assume $a,b \neq 0$ are integers that do not depend on $t$.
	\end{hypothesis}
	
	Hypergeometric truncation relationships similar to the equivalence predicted in Hypothesis~\ref{hyp:hypergeometricTruncation} have been established in many contexts. The groundbreaking work of \cite{CORV, CORV2} established such a relationship for the Fermat quintic pencil and the mirror quintic pencil. In the toric setting, Kadir studied a two-parameter family of Calabi-Yau hypersurfaces and their mirrors in \cite{kadir}. The first and second author gave hypergeometric formulas for the Hasse-Witt invariants of certain K3 surface pencils and their mirrors realized as toric hypersurfaces in \cite{SW}. For a detailed study of hypergeometric point counts on K3 surfaces or Calabi-Yau varieties in the case of the Berglund--H\"ubsch--Krawitz mirror symmetry construction, see \cite{supersquare1, supersquare2}.
	
	We test Hypothesis~\ref{hyp:hypergeometricTruncation} by computing $\# X_t^{2,4}(\mathbb{F}_p)$ for small $p$ and integer $t$. In order to verify our computations, we compute these numbers in two ways. Our first method is by direct substitution. We stratify the Grassmannian $G(2, 4)$ by Schubert cells using the normalization of \cite[Example 2.10]{kresch}, enumerate all of the points on $G(2,4)$ over $\mathbb{F}_p$ for small $p$, and count the number of these points on $X_t^{2,4}$ for $t=1, \dots, p-1$. (We omit $t=0$ as this does not correspond to a smooth Calabi-Yau threefold.) There are $(p^2+1)(p^2+p+1)$ points on $G(2,4)$ over $\mathbb{F}_p$, so this method is highly computationally intensive, but even for very small $p$ the results are enlightening; we will take $p=5, 7$, and $11$ as running examples in the following discussion.
	
	Our second method uses a result of \cite{HLYY} that generalizes an algorithm of Katz. Inspired by work of Rietsch in \cite{rietsch}, the authors of \cite{HLYY} expand the periods of $X_t^{r,n}$ in local coordinates given
	by a Richardson variety corresponding to a resolution of
	singularities for a top-dimensional Schubert cell. (For Calabi-Yau hypersurfaces in $G(2,4)$, these coordinates are described in \cite[Example 3.3]{HLYY}.) By evaluating a truncated version of this expansion over a finite field, we obtain the point count modulo $p$.
	
	\begin{example}
		Let $A_t$ be the polynomial defining the arrow pencil $X_t^{2,4}$ and let $B_t = A_t/(p_{12} p_{14} p_{23} p_{34})$, that is, the quotient of $A_t$ by the product of the frozen variables. In the coordinates $t_1, \dots, t_4$ of  \cite[Example 3.3]{HLYY}, $B_t$ is given by:
		
	\[	-\frac{{\left({\left({\left(\frac{1}{t_{1}^{2} t_{2} t_{3}} - \frac{t_{1} + t_{4}}{t_{1}^{2} t_{2} t_{3} t_{4}}\right)}^{4} + \frac{1}{t_{1}^{4}} + \frac{1}{t_{1}^{4} t_{2}^{4}} + \frac{1}{t_{1}^{4} t_{3}^{4}} + \frac{{\left(t_{1} + t_{4}\right)}^{4}}{t_{1}^{4} t_{2}^{4} t_{3}^{4} t_{4}^{4}} + 1\right)} t - \frac{\frac{1}{t_{1}^{2} t_{2} t_{3}} - \frac{t_{1} + t_{4}}{t_{1}^{2} t_{2} t_{3} t_{4}}}{t_{1}^{2} t_{2} t_{3}}\right)} t_{1}^{2} t_{2} t_{3}}{\frac{1}{t_{1}^{2} t_{2} t_{3}} - \frac{t_{1} + t_{4}}{t_{1}^{2} t_{2} t_{3} t_{4}}}. \]
	
	We take the Taylor series for $1/B_t$, expanding about $t=0$. For $k=0, \dots, p-1$, we consider the constant term $c_k$ of the coefficient of $t^k$ in this Taylor series. Multiplying $c_k$ by $t^k$, we obtain a new series $1+12 t^2+492 t^4+ 32880 t^6+ 2743020 t^8+ 257986512 t^{10} + \dots$. (Note that only even powers of $t$ appear.) The corresponding truncated series $c_0 + \dots + c_{k-1} t^{k-1}$, evaluated modulo $p$, yields a formula for the Hasse-Witt invariant $\mathrm{HW}_p(X_t^{2,4})$, and we have the relationship $1-\mathrm{HW}_p(X_t^{2,4}) \equiv \# X_t^{2,4}(\mathbb{F}_p) \pmod{p}$.
	\end{example}

We illustrate the resulting point counts for $p=5, 7$, and $11$ in Tables~\ref{Ta:5}, \ref{Ta:7}, and \ref{Ta:11}.
	
	\begin{table}[htb]
	    \begin{tabular}{|c|c|c|} 
	    	\hline
	    	$t$ & $\# X_t^{2,4}(\mathbb{F}_p)$ & $\# X_t^{2,4}(\mathbb{F}_p) \pmod{p}$ \\
	    	\hline 
	    	1 & 296 & 1 \\
	    	2 & 320 & 0 \\
	    	3 & 320 & 0 \\
	    	4 & 296 & 1 \\
		\hline
		\end{tabular}
	\caption{$p=5$}
	\label{Ta:5}
\end{table}

	\begin{table}[htb]
	\begin{tabular}{|c|c|c|} 
		\hline
		$t$ & $\# X_t^{2,4}(\mathbb{F}_p)$ & $\# X_t^{2,4}(\mathbb{F}_p) \pmod{p}$ \\
		\hline 
	1 & 384 & 6 \\
	2 & 388 & 3\\
	3 & 352 & 2\\
	4 & 520 & 2\\
	5 &  416  & 3\\
	6 & 384 & 6\\
		\hline
	\end{tabular}
	\caption{$p=7$}
	\label{Ta:7}
\end{table}
	
		\begin{table}[htb]
		\begin{tabular}{|c|c|c|} 
			\hline
			$t$ & $\# X_t^{2,4}(\mathbb{F}_p)$ & $\# X_t^{2,4}(\mathbb{F}_p) \pmod{p}$ \\
			\hline 
			1 & 1280& 4 \\
			2 & 1392& 6 \\
			3 & 1380& 5 \\
			4 & 1712& 7 \\
			5 &1536& 7 \\
			6& 1536& 7 \\
			7& 1536& 7 \\
			8 &1424& 5 \\
			9& 1216& 6 \\
			10& 1544& 4 \\
			\hline
		\end{tabular}
		\caption{$p=11$}
		\label{Ta:11}
	\end{table}
	
	Comparing to truncations of hypergeometric series, we observe that Hypothesis~\ref{hyp:hypergeometricTruncation} does not hold for these values of $p$.
	
	\begin{proposition}\label{prop:hypothesisFail}
		No truncation relationship of the form described in Hypothesis~\ref{hyp:hypergeometricTruncation} holds for $p=5, 7$, or $11$.
	\end{proposition}

	\subsection{Griffiths ring for complete intersections}
	
	Proposition~\ref{prop:hypothesisFail} shows that the properties of the $X_t^{2,4}$ arrow pencil and thus the candidate mirror $Y_t^{2,4}$  differ from the classical mirror of Grassmannian hypersurfaces in $G(2,4)$. In this section, we explain this discrepancy by computing the number of deformations of complex structure of $Y_t^{2,4}$. We will use the fact that Calabi-Yau hypersurfaces in $G(2,4)$ can be realized as complete intersections in projective space.
	
	The cohomology of complete intersections in projective space has been extensively studied. Given a smooth, transversal complete intersection $X$ in $\mathbb{P}^{m-1}$, we define its \emph{primitive cohomology} as $\mathrm{Coker}(\iota^*: H^*(\mathbb{P}^{m-1}) \to H^*(X))$, 
	where $\iota: X \hookrightarrow \mathbb{P}^{m-1}$ is the inclusion map. The authors of \cite{konno, terasoma, dimca} described a method for realizing the primitive cohomology of complete intersections using residues, which
	\cite{mavlyutov} generalized further to complete intersections in simplicial toric varieties. Adolphson and Sperber gave a ring-theoretic treatment with an eye toward arithmetic applications in \cite{AS}.
	In the physics literature, the authors of \cite{LT} studied the cohomology of Calabi-Yau complete intersections in projective spaces; the dissertation \cite{peternell} elaborates on this method. 
	
	\begin{definition}[\cite{konno, terasoma, dimca}]
		The \emph{Jacobian ideal} of a smooth transversal complete intersection $X$ in $\mathbb{P}^{m-1}$ given by polynomials $f_1, \dots, f_c$ in homogeneous coordinates is the ideal $J_{f_1, \dots, f_c}$ in $\C[x_1, \dots, x_m, y_1, \dots, y_c]$ generated by the polynomials $f_1, \dots, f_c$ and $\frac{\partial F}{\partial x_1}, \dots, \frac{\partial F}{\partial x_m}$, where $F = y_1 f_1 + \dots + y_c f_c$.
		
		The \emph{Jacobian ring} or \emph{Griffiths ring} $R_{f_1, \dots, f_c}$ is the quotient 
		$$\C[x_1, \dots, x_m, y_1, \dots, y_c]/J_{f_1, \dots, f_c}.$$
	\end{definition}
	
	(Note that $\frac{\partial F}{\partial y_j} = f_j$, so the definition of the Jacobian ideal is symmetric in the $x$- and $y$-variables.)
	
	Write $d_j$ for the degree of $f_j$ in $\C[x_1, \dots, x_m]$. We may induce a bigrading on the Jacobian ring by setting $\deg(x_i) = (1,0)$ and $\deg(y_j) = (-d_j,1)$. We write $R_{f_1, \dots, f_c}^{(a,b)}$ for the $(a,b)$-bigraded piece of the Jacobian ring.
	
	\begin{theorem}[\cite{konno, terasoma, dimca}]
		\[R_{f_1, \dots, f_c}^{(d_1+\dots+d_c-m,q)} \cong H_{\mathrm{prim}}^{(m-c-q-1,q)}(X).\]
	\end{theorem}
	
	Now, let $X$ be a smooth Calabi-Yau threefold realized as a quartic hypersurface in $G(2,4)$. View $X$ as a hypersurface in $\mathbb{P}^{5}$ determined by a quartic and the quadratic Pl\"ucker relation. We have $m=6$ and $c=2$. Thus $H_{\mathrm{prim}}^{(3,0)}(X)$ is isomorphic to the elements of the Jacobian ring $R_{f_1, f_2}$ of bigrading $(4+2-6,0)=(0,0)$, that is, the elements generated by constants. Similarly, $H_{\mathrm{prim}}^{(2,1)}(X)$ is isomorphic to the elements of the Jacobian ring $R_{f_1,  f_2}$ of bigrading $(0,1)$. Such elements are generated by monomials that have degree 4 in the Pl\"ucker coordinates (which we identify with the $x$-variables) and degree 1 in the $y$-variables. 
	
	To work with elements of the Jacobian ring, we use Gr\"obner basis computations in the computer algebra system Magma (cf. \cite{magma}) over the ring $\mathbb{Q}(t)[p_{12}, \dots, p_{34}, y_1, y_2]$.
	
	\begin{example}
	For the arrow pencil $X^{2,4}_t$, $R_{f_1,  f_2}^{(0,0)}$ is a one-dimensional vector space generated by any constant, and $R_{f_1,  f_2}^{(0,1)}$ is an 89-dimensional vector space. Thus, we have checked that $h^{3,0}(X^{2,4}_t) = 1$ and $h^{2,1}(X^{2,4}_t) = 89$, as befits a smooth Calabi-Yau threefold in $G(2,4)$.
	\end{example}
	
Next we study the $H_{4,2}$-invariant subspace of $H^{2,1}(X^{2,4}_t)$. We do so by computing generators for $H_{4,2}$ and checking the induced action on monomials with degree 4 in the Pl\"ucker coordinates and degree 1 in the $y$-variables.
	
	\begin{example}\label{ex:g24fixed}
		$(H^{2,1}(X^{2,4}_t))^{H_{4,2}}$ is a 5-dimensional vector space spanned by elements corresponding to the $H_{4,2}$-invariant monomials
		
		\[y_1 p_{34}^4, y_1 p_{14}^2 p_{23}^2, y_1 p_{13}^2 p_{24}^2, y_1 p_{24}^4, y_1 p_{23}^4.\] 
	\end{example}
	
But 	$(H^{2,1}(X^{2,4}_t))^{H_{4,2}} \cong H^{2,1}(Y^{2,4}_t)$. We have thus established a counterexample to Conjecture~\ref{Con:CDKmirror}:
	
	\begin{counterexample}\label{ce:g24}
		The pencil $Y^{2,4}_t$ has $h^{2,1} = 5$, while the classical mirror to Calabi-Yau threefold hypersurfaces in $G(2,4)$ has $h^{2,1} = 1$.
	\end{counterexample}
	
	\section{Middle cohomology of arrow pencils}\label{S:griffithsRing}
	
	The techniques used to establish Counterexample~\ref{ce:g24} depended on realizing members of the arrow pencil $X^{2,4}_t$ as complete intersections. In order to study higher-dimensional examples, we must develop techniques appropriate for more general Grassmannian hypersurfaces.
	
	\subsection{Griffiths ring for Grassmannian hypersurfaces}
	
	Set $N = r(n-r)-1$. Our goal is to understand the dimension of the vector space $(H^{N-1,1}(X^{r,n}_t))^{H_{n,r}}$: this measures the size of the space of complex deformations of  the arrow pencil $X^{r,n}_t$. 
	
	We will use a construction of a Griffiths ring for Grassmannian hypersurfaces described by Fatighenti and Mongardi in  \cite{FM}. This construction is phrased in terms of the vanishing cohomology. For smooth hypersurfaces in projective space, the notions of vanishing cohomology and primitive cohomology coincide. In our setting, we have a smooth Calabi-Yau hypersurface $X$ embedded in a Grassmannian $G(r,n)$ via a map $\iota$. The pushforward $\iota_*: H_{N-1}(X) \to H_{N-1}(G(r,n))$ on homology induces a pushforward map $\iota_*: H^{N-1}(X) \to H^{N+1}(G(r,n))$ on cohomology. The \emph{vanishing cohomology} $H^*_{\mathrm{van}}$ is the kernel 
	$$\mathrm{Ker}(\iota_*: H^{N-1}(X) \to H^{N+1}(G(r,n)))$$ 
	of this map. 
	
	A different subspace of cohomology, the \emph{hyperplane cokernel}, will take the role played by the subspace spanned by a hyperplane class in projective space.
	
	\begin{definition}[\cite{FM}, Definition 3.7]
		Let $j \leq \frac{N+1}{2}$. The \emph{hyperplane cokernel} $C_{j-1,j}$ is the cokernel of the injective map on the cohomology of the Grassmannian $G(r,n)$ given by multiplication by the hyperplane class:
		\[0 \to H^{j-1,j-1}(G(r,n)) \to H^{j,j}(G(r,n)).\]
	\end{definition}
	
	We wish to define an ideal that will take the place of the Jacobian ideal for hypersurfaces in projective space. Let $\mathcal{I}$ be the collection of $r$-tuples $(i_1, \dots, i_r)$ of elements in $\{1, \dots, n\}$ such that $i_1 < i_2 < \dots < i_r$ and fix $n$ variables $x_1, \dots, x_n$. One may identify the Pl\"ucker coordinate $p_{i_1\dots i_r}$ with the wedge product $x_{i_1} \wedge \dots \wedge x_{i_r}$. Let $P$ be the ideal generated by the Pl\"ucker relations. Let $S^{r,n}$ be the coordinate ring of the (affine cone over) the Grassmannian $G(r,n)$; we identify $S^{r,n}$ with $\mathbb{C}[p_I]_{I \in \mathcal{I}}/P$. Define derivations $D^i_j$ on $\C[x_1, \dots, x_n]$ by $D^i_j = x_i \frac{\partial}{\partial x_j}$. The derivations $D^i_j$ extend to wedge products of $x_1, \dots, x_n$ by the Leibniz rule. We then have an induced action on Pl\"ucker coordinates given by
	\begin{equation}
		D^i_j \, p_{i_1\dots i_r} =
		\begin{cases}
			0 & \mbox{if }j \notin \{i_1, \dots, i_r\}\mbox{ or } \{i,j\} \subset \{i_1, \dots, i_r\}\\
			\pm p_{i_1\dots i \dots \hat{j} \dots i_r} & \mbox{if }i \notin \{i_1, \dots, i_r\}\mbox{ and }j \in \{i_1, \dots, i_r\}.
		\end{cases}
	\end{equation}
	Here $\hat{j}$ indicates that this term is omitted. The choice of sign for $\pm p_{i_1\dots i \dots \hat{j} \dots i_r}$ is positive if an even number of swaps are needed to place $x_{i_1} \wedge \dots \wedge x_i \widehat{x_{j}} \wedge \dots \wedge x_{i_r}$
	in increasing order by index, and negative otherwise. Applying the Leibniz rule once again, we obtain an induced action of the $D^i_j$ on $\mathbb{C}[p_I]_{I \in \mathcal{I}}$ and thus on $S^{r,n}$.
	
	\begin{definition}[\cite{FM}, Definition 3.2]
		The \emph{generalized Jacobian ideal} or \emph{Griffiths ideal} of a smooth hypersurface given by a polynomial $f$ in Pl\"ucker coordinates in the Grassmannian $G(r,n)$ is the ideal $J_f$ of $S^{r,n}$ generated by the image of $f$ and the set of polynomials given by
		$$\{D^i_j \, f, D^i_i \, f - D^j_j \,f \mid i, j = 1, \dots, n, i \neq j \}.$$
		The \emph{(generalized) Jacobian ring} or \emph{(generalized) Griffiths ring} $R_f$ is the quotient $S^{r,n}/J_f$.
	\end{definition}
	
	We write $[R_f]_m$ for the vector space spanned by elements of $R_f$ represented by monomials of degree $m$.
	
	We may now give a description of the vanishing cohomology of a Grassmannian hypersurface in terms of the generalized Jacobian ring:
	
	\begin{theorem}[\cite{FM}, Theorem 1.1]
		Let $X$ be a smooth hypersurface in the Grassmannian $G(r,n)$ given by a polynomial $f$ of degree $d$ in Pl\"ucker coordinates, and assume $d \geq n-1$. If $N=\dim X$ is even, then
		\[[R_f]_{(q+1)d-n} \cong H^{N-q,q}_\mathrm{van}(X, \mathbb{C}). \]
		If $N=\dim X$ is odd, then
		\[[R_f]_{(q+1)d-n} \cong H^{N-q,q}_\mathrm{van}(X, \mathbb{C}) \oplus \delta_{q, \frac{N+1}{2}} C_{q-1,q}. \]
		Here $\delta_{q, \frac{N+1}{2}}$ is the Kronecker delta.
	\end{theorem}
	
	\subsection{Arrow pencils and arrow pencil alternatives}
	
	We now apply the generalized Jacobian ring framework to study arrow pencils and related families. We begin by reconsidering Example~\ref{ex:g24fixed} using the generalized Jacobian ideal framework. In order to do so, we use the SageMath computer algebra system (\cite{sage}) to compute the generalized Jacobian ideal, followed by Gr\"obner basis computations in Magma over the ring $\mathbb{Q}(t)[p_{12}, \dots, p_{34}]$. We conclude:
	
	\begin{example}\label{ex:g24fixed2}
		$(H^{2,1}(X^{2,4}_t))^{H_{4,2}}$ is a 5-dimensional vector space spanned by elements corresponding to the $H_{4,2}$-invariant monomials
		
		\[p_{34}^4, p_{14}^2 p_{23}^2, p_{13}^2 p_{24}^2, p_{24}^4, p_{23}^4.\] 
	\end{example}

One question that naturally arises is whether one may modify the arrow pencil in some way to obtain a family with the desired Hodge invariants. Each of the monomials in the arrow pencil $X^{2,4}_t$ is invariant under the action of the group $H_{4,2}$, but these are not the only monomials of degree 4 in the Pl\"ucker coordinates invariant under the group action. The complete list (in lexicographic order) is $\mathcal{M} :=$ $\{p_{12}^{4}, p_{12}^{2} p_{34}^{2}$, $p_{12} p_{13} p_{24} p_{34}$, $p_{12} p_{14} p_{23} p_{34}$, $p_{13}^{4}, p_{13}^{2} p_{24}^{2}$, $p_{13} p_{14} p_{23} p_{24}$, $p_{14}^{4}, p_{14}^{2} p_{23}^{2}$, $p_{23}^{4}, p_{24}^{4}, p_{34}^{4}\}$. Recall that the arrow pencil $X^{2,4}_t$ is defined by deforming the product of the frozen variables, which corresponds to the point of maximally unipotent monodromy, by the product of our parameter $t$ and the sum of $4$th powers of frozen variables. We may thus subdivide the list of $H(2,4)$-fixed monomials into pure fourth powers, the pairs of squares $\{p_{12}^{2} p_{34}^{2}, p_{13}^{2} p_{24}^{2}, p_{14}^{2} p_{23}^{2}\}$, and the products of four distinct monomials $\{p_{12} p_{13} p_{24} p_{34}, p_{12} p_{14} p_{23} p_{34}, p_{13} p_{14} p_{23} p_{24}\}$. Modifying the deforming polynomial in the arrow pencil by adding either pairs of squares, the non-frozen products of four distinct monomials, or both, we obtain four pencils (including the arrow pencil itself).

Our next theorem (a restatement of our \hyperref[T:main]{Main Theorem}) shows that all of these $H_{4,2}$-invariant pencils have a similar piece of cohomology fixed by the group action.
	
\begin{theorem}
	Let $Z^{2,4}_t$ be any of the following $H_{4,2}$-invariant pencils of Calabi-Yau threefolds in the Grassmannian $G(2,4)$: 
	
	\begin{align*}
	t \left( p_{12}^4 + \cdots + p_{34}^4  \right) + p_{12} p_{23} p_{34} p_{14} & = 0 \\
	t \left( p_{12}^{4}  + \cdots + p_{34}^4  +  p_{14}^{2} p_{23}^{2} +   p_{13}^{2} p_{24}^{2} +  p_{12}^{2} p_{34}^{2} \right) + p_{12} p_{14} p_{23} p_{34} &= 0\\
	t \left( p_{12}^{4} + \cdots + p_{34}^4 +  p_{13} p_{14} p_{23} p_{24} + p_{12} p_{13} p_{24} p_{34}\right) + p_{12} p_{14} p_{23} p_{34} & = 0 \\
	t \left(p_{12}^{4} + \cdots + p_{34}^4 + p_{14}^{2} p_{23}^{2} +  p_{13} p_{14} p_{23} p_{24} + p_{13}^{2} p_{24}^{2} + p_{12} p_{13} p_{24} p_{34} +  p_{12}^{2} p_{34}^{2} \right) & \\ 
	+ p_{12} p_{14} p_{23} p_{34} & = 0.
\end{align*}
	
		Then $(H^{2,1}(Z^{2,4}_t))^{H_{4,2}}$ is a 5-dimensional vector space.
\end{theorem}

\begin{proof}
	The proof follows by using the SageMath computer algebra system (\cite{sage}) to compute the generalized Jacobian ideal, then using Gr\"obner basis computations in Magma over the ring $\mathbb{Q}(t)[p_{12}, \dots, p_{34}]$ to compute the subspace of the generalized Jacobian ring spanned by the $H_{4,2}$-invariant degree 4 monomials $\mathcal{M}$ in each case. For the pencils other than the arrow pencil, this subspace is spanned by the monomials $p_{23}^4$, $p_{13}p_{14}p_{23}p_{24}$, $p_{34}^4$, $p_{14}^2 p_{23}^2$,
	$p_{24}^4$, which are distinct in the generalized Jacobian ring in each case.	
\end{proof}
	
	Next, we consider Calabi-Yau fivefolds realized as hypersurfaces in $G(2,5)$. Gr\"obner basis computations over a ring with coefficients in $\mathbb{Q}(t)$ are highly computationally expensive. Instead of working directly with this ring, we compute $(H^{2,1}(X^{2,5}_t))^{H(2,5)}$ for arbitrarily chosen specific values of $t$. We find:
	
	\begin{example}\label{ex:g25fixed}
		For $t=2, 3, 7, 13$, $(H^{2,1}(X^{2,5}_t))^{H(2,5)}$ is an 11-dimensional vector space spanned by elements corresponding to the $H(2,5)$-invariant monomials 
		
		\begin{align*}
			p_{24}^5, p_{35}^5, p_{14}p_{15}p_{23}p_{24}p_{35}, p_{15}^2p_{23}p_{24}p_{34}, p_{13}p_{14}p_{25}^2p_{34}, \\
			p_{23}^5, p_{25}^5, p_{34}^5, p_{13}p_{15}p_{24}p_{25}p_{34}, p_{45}^5, p_{14}^2p_{23}p_{25}p_{35}.
		\end{align*}
	\end{example}

This computation suffices to yield the following counterexample:
	
		\begin{counterexample}\label{ce:g25}
		The pencil $Y^{2,5}_t$ is not a pencil of Calabi-Yau fivefolds with smooth members all satisfying $h^{2,1} = 1$.
	\end{counterexample}
	
Note that it can be hard to generate examples of Calabi-Yau $n$-folds where $h^{n-1,1}$ is small. For example, the smallest $h^{n-1,1}$ for Calabi-Yau fivefolds realized as complete intersections in some $\mathbb{P}^k$ is 289, which occurs in the case of complete intersections of 6 quadrics in $\mathbb{P}^{11}$. (see  \cite[Table 5]{HLS}). Thus, the strategy of studying Calabi-Yau varieties realized as quotients of highly symmetric pencils may yield new examples of Calabi-Yau $n$-folds with small $h^{n-1,1}$ of geographical interest.


\begin{thebibliography}{DKSSVW18a}
	
	\bibitem[AS06]{AS}
	Adolphson, Alan and Sperber, Steven: On the {J}acobian ring of a complete intersection. J. Algebra 304, 1193--1227 (2006).
	
	\bibitem[AS16]{ASgkz} Adolphson, A., Sperber, S.: {$A$}-hypergeometric series and the {H}asse-{W}itt matrix of a
	hypersurface, Finite Fields Appl. \textbf{41}, 55--63 (2016).
	
	\bibitem[BCFKvS98]{conifold}
	Batyrev, Victor V., Ciocan-Fontanine, Ionu\c{t}, Kim, Bumsig, and van Straten, Duco: Conifold transitions and mirror symmetry for {C}alabi-{Y}au complete intersections in {G}rassmannians. Nuclear Phys. B 514, 640--666 (1998).
	
	\bibitem[BV21a]{BVi}
	Beukers, F., Vlasenko, M.: Dwork crystals {I}, Int. Math. Res. Not. IMRN 12, 8807--8844 (2021).
	
	\bibitem[BV21b]{BVii}
	Beukers, F., Vlasenko, M.: Dwork crystals {II}, Int. Math. Res. Not. IMRN 6, 4427--4444 (2021).
	
	\bibitem[BV23]{BViii}
	Beukers, F., Vlasenko, M.: Dwork crystals {III}: {F}rom {E}xcellent {F}robenius {L}ifts
	{T}owards {S}upercongruences, Int. Math. Res. Not. IMRN 23, 20433--20483 (2023).

\bibitem[BCP97]{magma}
Bosma, Wieb, Cannon, John, and Playoust, Catherine: The {M}agma algebra system. {I}. {T}he user language, J. Symbolic Comput. 24, 235--265 (1997).

\bibitem[CDK21]{CDK} 
Coates, T., Doran, C., and Kalashnikov, E.: Unwinding toric degenerations and mirror symmetry for Grassmannians. Forum Math. Sigma 10, Paper No. e111, 33 (2022).

\bibitem[CDRV00]{CORV}
Candelas, P., de la Ossa, X., Rodr\'{i}guez Villegas, F.: Calabi--Yau manifolds over finite fields, {I}. \url{https://arxiv.org/abs/hep-th/0012233} (2000).  

\bibitem[CDRV01]{CORV2} 
Candelas, P., de la Ossa, X., Rodr\'{i}guez Villegas, F.: Calabi--Yau manifolds over finite fields, {II}.  In: Calabi--Yau varieties and mirror symmetry.  Fields Inst. Commun., vol. 38, 121-157.  Amer. Math. Soc., Providence (2003).
%
\bibitem[Dim95]{dimca}
Dimca, Alexandru: Residues and cohomology of complete intersections. Duke Math. J. \textbf{78}, 89--100 (1995).

\bibitem[DKSSVW18]{supersquare1}
Doran, C.F., Kelly, T.L., Salerno, A., Sperber, S., Voight, J., Whitcher, U.: Zeta functions of alternate mirror Calabi-Yau families.  Israel J. Math. \textbf{228}, no. 2, 665--705 (2018).

\bibitem[DKSSVW20]{supersquare2}
Doran, C.F., Kelly, T.L., Salerno, A., Sperber, S., Voight, J., Whitcher, U.: Hypergeometric decomposition of symmetric K3 quartic pencils. Res. Math. Sci. \textbf{7}, article 7 (2020).

\bibitem[Dwo69]{padic} 
Dwork, B., {$p$}-adic cycles, Inst. Hautes \'Etudes Sci. Publ. Math. \textbf{37}, 27--115 (1969).


\bibitem[FM21]{FM}
Fatighenti, Enrico and Mongardi, Giovanni: A note on a {G}riffiths-type ring for complete intersections	in {G}rassmannians. Math. Z. 299, 1651--1672 (2021).

\bibitem[GP90]{GP}
Greene, B. R. and Plesser, M. R.: Duality in \{C\}alabi-\{Y\}au moduli space. Nuclear Phys. B 338 (1990), 15--37.

\bibitem[HLS09]{HLS}
Haupt, Alexander S., Lukas, Andre and Stelle, K. S.: M-theory on {C}alabi-{Y}au five-folds. J. High Energy Phys. 5, 069 (2009).

\bibitem[HLYY22]{HLYY}
Huang, A., Lian, B., Yau, S.-T., Yu, C.: Hasse-Witt matrices, unit roots and period integrals. Math. Ann. (2022).

\bibitem[Kad06]{kadir} 
Kadir, S.: Arithmetic mirror symmetry for a two-parameter family of {C}alabi-{Y}au manifolds, in Mirror symmetry. {V}, AMS/IP Stud. Adv. Math., \textbf{38}, 35--86, Amer. Math. Soc., Providence, RI (2006).

\bibitem[Kat73]{katzCongruence}
Katz, N.M.: Une formule de congruence pour la fonction $\zeta$, S.G.A. 7 II, Lecture Notes in Mathematics
\textbf{340}, Springer (1973).

\bibitem[Kon91]{konno}
Konno, Kazuhiro: On the variational {T}orelli problem for complete intersections. Compositio Math. \textbf{78}, 271--296 (1991).

\bibitem[Kre07]{kresch}
Kresch, A.: Flag varieties and {S}chubert calculus, \emph{Algebraic groups} 73--86, Universit\"{a}tsverlag G\"{o}ttingen, G\"{o}ttingen (2007).

\bibitem[Lib96]{libgober}
Libgober, A.: Differential forms on complete intersections and a related quotient module. Proceedings of the {H}irzebruch 65 {C}onference on {A}lgebraic {G}eometry ({R}amat {G}an, 1993). Israel Math. Conf. Proc. 9, 295--305. Bar-Ilan Univ., Ramat Gan (1996).

\bibitem[LT93]{LT}
Libgober, A., Teitelbaum, J.: Lines on {C}alabi-{Y}au complete intersections, mirror symmetry, and {P}icard-{F}uchs equations. Internat. Math. Res. Notices. No. 1, 29--39 (1993).

\bibitem[Mav99]{mavlyutov}
Mavlyutov, Anvar R.: Cohomology of complete intersections in toric varieties. Pacific J. Math. \textbf{191}, 133--144 (1999).

\bibitem[Pet18]{peternell}
Peternell, Kathrin Natalie: Coherent Sheaves on Calabi-Yau manifolds, Picard-Fuchs equations and potential functions. Dissertation, Albert Ludwigs University of Freiburg, 2018.


\bibitem[Rie08]{rietsch}
Rietsch, Konstanze: A mirror symmetric construction of {$qH^\ast_T(G/P)_{(q)}$}. Adv. Math. 217, 2401--2442 (2008).

\bibitem[S{\etalchar{+}}24]{sage}
\emph{{S}ageMath, the {S}age {M}athematics {S}oftware {S}ystem ({V}ersion
	10.1)}, The Sage Developers, 2024, {\tt https://www.sagemath.org}.

\bibitem[SW22]{SW}
Salerno, Adriana and Whitcher, Ursula: Hasse-Witt matrices and mirror toric pencils. Advances in Theoretical and Mathematical Physics Vol. 26, no.9, 3345-3375 (2022).

\bibitem[SWY24]{code}
Salerno, A., Whitcher, U., and Yu, C. Deformations of highly symmetric Calabi-Yau Grassmannian hypersurfaces (computational tools). Zenodo (2024). \url{https://doi.org/10.5281/zenodo.10570227}

\bibitem[Ter90]{terasoma}
Terasoma, Tomohide: Infinitesimal variation of {H}odge structures and the weak global {T}orelli theorem for complete intersections. Ann. of Math. (2) \textbf{132}, 213--235 (1990).

\bibitem[Vla18]{vlasenko} 
Vlasenko, M.: Higher Hasse–Witt matrices, Indag. Math. \textbf{29} 1411–1424 (2018).

\end{thebibliography}
\end{document}